\documentclass{article}
\usepackage{amsfonts}
\usepackage{amsmath}
\usepackage[mathscr]{eucal}
\usepackage{amssymb}
\usepackage{latexsym}
\usepackage{amsthm}
\usepackage{amscd}
\usepackage{epsf}
\usepackage{graphicx}
\usepackage{enumerate}
\usepackage{cases}
\usepackage{url}
\usepackage{algpseudocode, algorithm}

\newtheorem{theorem}{Theorem}[section]

\newtheorem{lemma}[theorem]{Lemma}

\theoremstyle{definition}
\newtheorem{definition}[theorem]{Definition}

\theoremstyle{remark}
\newtheorem{remark}[theorem]{Remark}

\newcommand{\bsa}{\boldsymbol{a}}
\newcommand{\bsb}{\boldsymbol{b}}
\newcommand{\bsc}{\boldsymbol{c}}

\newcommand{\bsh}{\boldsymbol{h}}
\newcommand{\bsk}{\boldsymbol{k}}

\newcommand{\bsu}{\boldsymbol{u}}
\newcommand{\bsv}{\boldsymbol{v}}

\newcommand{\bsx}{\boldsymbol{x}}

\newcommand{\bsalpha}{\boldsymbol{\alpha}}
\newcommand{\bsbeta}{\boldsymbol{\beta}}
\newcommand{\bsgamma}{\boldsymbol{\gamma}}

\newcommand{\bszero}{\boldsymbol{0}}

\newcommand{\bN}{\mathbb{N}}
\newcommand{\bZ}{\mathbb{Z}}
\newcommand{\bQ}{\mathbb{Q}}
\newcommand{\bR}{\mathbb{R}}

\newcommand{\Rb}[4]{\phi_{#1}(#2,#3,#4)}
\newcommand{\Rc}[4]{\psi_{#1}(#2,#3,#4)}
\newcommand{\UPa}[3]{\tau_{#1}(#2,#3)}
\newcommand{\ave}[2]{\rho_{#1}(#2)}
\newcommand{\Generatorshort}[3]{\mathcal{P}_{#1}(#2,#3)}
\newcommand{\Aset}[1][i]{\mathcal{A}_{#1}}
\newcommand{\Bset}[1][i]{\mathcal{B}_{#1}}

\newcommand{\bbX}{\mathbb{X}}
\newcommand{\Boxineq}[4]{#2 \leq A_{#1} #3 \leq #4}

\newcommand{\bsalphas}[2]{\bsalpha_{#1, #2}}
\newcommand{\bsbetas}[2]{\bsbeta_{#1, #2}}
\newcommand{\bsgammas}[2]{\bsgamma_{#1, #2}}
\newcommand{\bsxs}[2]{\bsx_{#1, #2}}

\newcommand{\abs}[1]{\lvert #1 \rvert}
\DeclareMathOperator{\diag}{diag}

\begin{document}

\title{Enumeration of the Chebyshev-Frolov lattice points in axis-parallel boxes\thanks{
The research of the authors was supported under the Australian Research Councils Discovery Projects funding scheme (project number DP150101770).
}}
\date{\today}
\author{Kosuke Suzuki\thanks{
School of Mathematics and Statistics, The University of New South Wales, Sydney, NSW 2052, Australia,
e-mail: kosuke.suzuki1@unsw.edu.au
}, 
Takehito Yoshiki\thanks{
School of Mathematics and Statistics, The University of New South Wales, Sydney, NSW 2052, Australia,
e-mail: takehito.yoshiki1@unsw.edu.au}}


\maketitle

\begin{abstract}
For a positive integer $d$,
the $d$-dimensional Chebyshev-Frolov lattice
is the $\bZ$-lattice in $\bR^d$ generated by the Vandermonde matrix
associated to the roots of the $d$-dimensional Chebyshev polynomial.
It is important to enumerate the points from the Chebyshev-Frolov lattices in axis-parallel boxes
when $d = 2^n$ for a non-negative integer $n$,
since the points are used for the nodes of Frolov's cubature formula,
which achieves the optimal rate of convergence for many spaces of functions 
with bounded mixed derivatives and compact support.
The existing enumeration algorithm for such points by Kacwin, Oettershagen and Ullrich
is efficient up to dimension $d=16$.
In this paper we suggest a new enumeration algorithm of such points for $d=2^n$, efficient up to $d=32$.
\end{abstract}

\section{Introduction}
Let $d$ be a positive integer and $\bbX \subset \bR^d$ be a $d$-dimensional lattice,
i.e., there exists an invertible $d \times d$ matrix $T$ over $\bR$ such that
\[
\bbX=T(\bZ^d)=\{T\bsk \mid \bsk \in \bZ^d\}.
\]
The lattice $\bbX$ is said to be admissible if 
\[
\min \left\{ \prod_{i=1}^d |x_i|  \; \middle| \; (x_1, \dots, x_d)^{\top} \in \bbX \setminus \{\bszero\} \right\}> 0.
\]
Using an admissible lattice $\bbX=T(\bZ^d)$,
Frolov's cubature formula approximates the integration value
\[
I(f) := \int_{[-1/2, 1/2]^d} f(\bsx) \, d\bsx
\]
of a function $f \colon [-1/2, 1/2]^d \to \bR$ by
\begin{equation}\label{eq:Frolov-Formula}
Q_{a^{-1}T}(f)
= \abs{\det(a^{-1}T)} \sum_{x \in a^{-1}\bbX \cap [-1/2, 1/2]^d} f(\bsx) \quad \text{for $a\geq1$}.
\end{equation}
Thus the nodes are the shrunk lattice points $a^{-1}\bbX$
inside the box $[-1/2,1/2]^d$.
Frolov's cubature formula is first proposed by Frolov \cite{Frolov1976ube} and 
has been studied in many papers, see
\cite{Dubinin1991oqf, Dubinin1997cfb, Dung2016hca, Krieg2015uam, Skriganov1994cud, Temlyakov1993apf, Temlyakov2003cfd, Ullrich2016mcm, Ullrich2016ueb, Ullrich2016rfc}.
One prominent feature of the formula is
that it achieves the optimal rate of convergence for various spaces of functions
with bounded mixed derivatives and compact support.
This means that the approximation is automatically good,
even without knowing specific information of integrands.
The constraint of compact supportness can be removed using some modification, see \cite{Nguyen2015cvs}.

The implementation of Frolov's cubature formula
requires to enumerate the points in the set $a^{-1}\bbX \cap [-1/2, 1/2]^d$,
or equivalently, the points in the set $\bbX \cap [-a/2, a/2]^d$.
However the enumeration is a difficult task even in moderate dimensions.
Recently, 
an efficient enumeration algorithm 
for the so-called Chebyshev-Frolov lattices up to $d=16$
was suggested by Kacwin, Oettershagen and Ullrich \cite{Kacwin2016ocf}.
Since the lattices are admissible when $d=2^n$,
it is possible to implement Frolov's cubature formula for $d=2^n$, up to $d=16$.
Numerical experiments based on the algorithm are given
in \cite{Kacwin2016rfc} and will be given in the forthcoming paper
by the authors of \cite{Kacwin2016ocf} and Mario Ullrich.
Our contribution in this paper is to suggest a new efficient enumeration algorithm
for the Chebyshev-Frolov lattices for $d=2^n$.
It is efficient up to $d=32$.

The Chebyshev-Frolov lattices for $d=2^n$ are examples of admissible lattices, suggested by Temlyakov \cite[IV.4]{Temlyakov1993apf}.
Let $P_d$ be a rescaled $d$-dimensional Chebyshev polynomial defined as
\begin{equation}\label{eq:Chebyshev-poly}
P_d(x) = 2 \cos(d \arccos(x/2)) \quad \text{for $|x|<2$}.
\end{equation}
Its roots are given by
\begin{equation}\label{eq:root-Chebyshev-usual}
\zeta_{n,k} = 2 \cos \left(\frac{\pi (2k-1)}{2d} \right), \quad k = 1, \dots, 2^n.
\end{equation}
With these roots, we define a Vandermonde matrix $T$ by
\[
T = (\zeta_i^{j-1})_{i, j =1}^d.
\]
Now the $d$-dimensional Chebyshev-Frolov lattice is defined as the lattice $T(\bZ^d)$.
It is known that the lattice $T(\bZ^d)$ is admissible if  and only if $d=2^n$.
This is a special case of a general construction method for admissible lattices for any $d$
elaborated in \cite{Temlyakov1993apf}, see also Section~\ref{sec:admissible}.
An advantage of the Chebyshev-Frolov lattices is that the generating matrices are explicitly given.

We now briefly recall results in \cite{Kacwin2016ocf}.
The paper established an enumeration algorithm of the lattice points in $[-a/2, a/2]^d$,
for any orthogonal lattices.
This is applicable to the Chebyshev-Frolov lattices as they are orthogonal.
Their experiment shows that it is efficient up to $d=16$.
They further proved properties of the Chebyshev-Frolov lattices summarized as follows.
\begin{theorem}[{\cite[Theorem~1.1]{Kacwin2016ocf}}]\label{thm:Kacwin}
For any positive integer $d$,
the $d$-dimensional Chebyshev-Frolov lattice $T(\bZ^d)$ is orthogonal.
In particular,
there exists a lattice representation $\tilde{T} = TS$ with some $S \in SL_d(\bZ)$ such that
\begin{itemize}
\item For each component $t_{i,j}$ of $\tilde{T}$, it holds that $|t_{i,j}| \leq 2$.
\item $\tilde{T}^\top \tilde{T} = diag(d, 2d, \dots, 2d)$.
\end{itemize}
\end{theorem}

Our algorithm is based on another property particular to the Chebyshev-Frolov lattices.
Our key observation is 
that the $2^n$-dimensional Chebyshev-Frolov lattice with a certain permutation of coordinates
is generated by a matrix $A_{n}$
which satisfies a recursive property as in \eqref{eq:definition-A}.
This property reduces the $2^n$-dimensional enumeration
to a number of $2^{n-1}$-dimensional enumerations as in Lemma~\ref{lem:recursive-domain}.
By applying this repeatedly, finally the enumeration is reduced to
nested 1-dimensional enumerations,
which can be implemented as $2^n$-folded for-loops.
We will expose them in Section~\ref{sec:algorithm}.
We will show that our algorithm is efficient up to $d=32$ in Section~\ref{sec:experiments}.

Another advantage of our algorithm is
that it can enumerate the Chebyshev-Frolov lattice points in arbitrary axis-parallel boxes.
This helps us to implement not only Frolov's cubature formula but also its randomization.
Randomized Frolov's cubature formula was introduced by Krieg and Novak \cite{Krieg2015uam}
and studied further by Ullrich \cite{Ullrich2016mcm}.
It inherits the prominent convergence behavior of the deterministic version as well as it is unbiased.
Further it also has the optimal order of convergence in the randomized sense
for Sobolev spaces with isotropic and mixed smoothness.
We will give how to enumerate the integration nodes of the deterministic and randomized versions
with our algorithm in Section~\ref{sec:Frolov}.

Throughout this paper we use the following notation.
The symbols $\bN$, $\bZ$, $\bQ$ and $\bR$ denote
the set of the non-negative integers, the integers, the rational numbers and the real numbers, respectively.
For $\bsx_1, \bsx_2 \in \bR^n$,
$(\bsx_1;\bsx_2) \in \bR^{2n}$ denotes the vector where $\bsx_1$ and $\bsx_2$ are vertically connected.
We denote $SL_d(\bZ)$ the special linear group of degree $d$ over $\bZ$,
i.e., the set of matrices over $\bZ$ whose determinant is 1.
For $x_1, \dots, x_d \in \bR$, 
$\diag(x_1, \dots, x_d)$ denotes the diagonal matrix with $(x_1, \dots, x_d)$ at the diagonal.
For a vector $\bsb = (b_1, \dots, b_d)^\top \in \bR^d$ and $\bsc= (c_1, \dots, c_d)^\top \in \bR^d$, we define
$[\bsb,\bsc] := \prod_{i=1}^d [b_i,c_i]$ and
$\max(\bsb, \bsc) := (\max(b_i, c_i))_{i=1}^d \in \bR^d$,
and denote $\bsb \leq \bsc$ if $b_i \leq c_i$ holds for all $1 \leq i \leq d$.

\section{Construction method of admissible lattices}\label{sec:admissible}

One general construction scheme for admissible lattices is studied in Temlyakov \cite[IV.4]{Temlyakov1993apf}.
Let $p_d(x) \in \bZ[x]$ be a $d$-dimensional polynomial with integer coefficients
satisfying the following three properties:
(i) its leading coefficient is 1,
(ii) it is irreducible over $\bQ$,
(iii) it has different $d$ real roots, say $\zeta_1, \dots, \zeta_d \in \bR$.
With these roots, we define a Vandermonde matrix $T$ by
\[
T = (\zeta_i^{j-1})_{i, j =1}^d.
\]
Then the lattice $T(\bZ)$ generated by $T$ is admissible.
Frolov used $q_d(x) = -1 + \prod_{j=1}^d (x -2j+1)$ in his paper \cite{Frolov1976ube}.
Note that he originally used the lattice made from $q_d(x)$ not for $T$ in \eqref{eq:Frolov-Formula} but for its dual lattice.
However, later it is shown that $T(\bZ^d)$ itself is admissible if and only if its dual lattice is admissible,
see \cite[Lemma~3.1]{Skriganov1994cud} and also \cite[Lemma~2.1]{Ullrich2016rfc} for a Vandermonde matrix.
One disadvantage of the choice of $q_d$ is that its roots are not given explicitly.

In \cite{Temlyakov1993apf} Temlyakov proposed to use the rescaled Chebyshev polynomials $P_d$ as in \eqref{eq:Chebyshev-poly}
when $d=2^n$ for a non-negative integer $n$.
It is shown that $P_d$ holds the condition (i) and (iii), and its roots are given as in \eqref{eq:root-Chebyshev-usual}.
Further $P_d$ is irreducible if and only if $d=2^n$.
Thus the Chebyshev-Frolov lattice, i.e., the lattice constructed as above with a use of  $P_d(x)$,
is admissible if and only if $d=2^n$.


\section{Enumeration of the Chebyshev-Frolov lattice points}\label{sec:algorithm}

\subsection{Recursive property of generating matrices}

Our considering Chebyshev-Frolov lattices are coordinate-permuted versions
of the usual ones.
Let $n \in \bN$ and put $d=2^n$.
We define $\sigma(n,k) \in \bZ$ for $1 \leq k \leq d$ recursively as 
$\sigma(0,1) = 1$ and
\[
\sigma(n+1,k) =
\begin{cases}
\sigma(n,k) & \text{if $1\leq k \leq d$}, \\
2d + 1 - \sigma(n, k-d) & \text{if $d +1 \leq k \leq 2d$}.
\end{cases}
\]
For all $n \in \bN$, the map $\sigma(n, \cdot)$ is a permutation on $\{1, \dots, d\}$,
which is shown by induction on $n$ as follows.
The case $n=0$ is trivial. We assume the lemma holds for $n$.
By the definition of $\sigma(n+1,k)$ and induction assumption,
$\sigma(n+1,\cdot)$ is a permutation on $\{1, \dots, d\}$
and also a permutation on $\{d+1, \dots, 2d\}$.
This proves the result for $n+1$.

We now define $\xi_{n,k} \in \bR$ as 
\[
\xi_{n,k} = 2 \cos \left(\frac{\pi(2 \sigma(n,k)-1)}{2d} \right) 
\quad \text{for $k = 1, \dots, d$},
\]
and consider a Vandermonde matrix $V_n \in \bR^{d}$ as
\[
V_n := (\xi_{n,i}^j)_{i,j =1}^{d} =
\begin{pmatrix}
1 & \xi_{n,1} & \cdots & \xi_{n,1}^{d-1} \\
1 & \xi_{n,2} & \cdots & \xi_{n,2}^{d-1} \\
\vdots & \vdots & \ddots & \vdots \\
1 & \xi_{n,d} & \cdots & \xi_{n,d}^{d-1}
\end{pmatrix}.
\]
Comparing $\xi_{n,k}$'s and $\zeta_{n,k}$'s defined as in \eqref{eq:root-Chebyshev-usual},
we find that $\xi_{n,k}$'s are also the roots of $P_d(x)$
since $\sigma(n, \cdot)$ is a permutation on $\{1, \dots, d\}$.
Thus the lattice $V_n(\bZ^d)$ is a coordinate permutation
of the usual Chebyshev-Frolov lattice.

Further we
define a diagonal matrix $D_n \in \bR^{d}$ as
\[
D_n := \diag(\xi_{n+1,1}, \dots, \xi_{n+1,d}).
\]
We are now ready to
define a matrix $A_n \in \bR^{d}$ recursively as
$A_0 = 1$ and
\begin{equation}\label{eq:definition-A}
A_{n+1} =
\begin{pmatrix}
A_n & D_n A_n \\
A_n & -D_n A_n
\end{pmatrix}.
\end{equation}
The following lemma shows that $A_n$ can be used as a generating matrix of the Chebyshev-Frolov lattices,
i.e., $V_n(\bZ^d) = A_n(\bZ^d)$.

\begin{lemma}\label{lem:matrix-recursive}
For all $n \in \bN$,
there exists $S_n \in \bZ^{2^n \times 2^n}$ such that $\det S_n = \pm 1$ and $V_n S_n = A_n$.
\end{lemma}

\begin{proof}
We prove the lemma by induction on $n$.
The case $n=0$ is trivial since $V_0 = A_0 = 1$.
Now we assume that the lemma holds for $n$ and show for $n+1$.
Put $d = 2^n$.
Define a matrix $V'_{n+1} \in \bR^{2d}$ obtained by column swapping of $V_{n+1}$ as
\[
V'_{n+1} =
\begin{pmatrix}
1 & \xi_{n+1,1}^2 & \cdots & \xi_{n+1,1}^{2(d-1)} & \xi_{n+1,1} & \xi_{n+1,1}^3 & \cdots & \xi_{n+1,1}^{2d-1}\\
1 & \xi_{n+1,2}^2 & \cdots & \xi_{n+1,2}^{2(d-1)} & \xi_{n+1,2} & \xi_{n+1,2}^3 & \cdots & \xi_{n+1,2}^{2d-1}\\
\vdots & \vdots & \ddots & \vdots & \vdots & \vdots & \ddots & \vdots \\
1 & \xi_{n+1,2d}^2 & \cdots & \xi_{n+1,2d}^{2(d-1)} & \xi_{n+1,2d} & \xi_{n+1,2d}^3 & \cdots & \xi_{n+1,2d}^{2d-1}\\
\end{pmatrix}.
\]
Since $V'_{n+1}$ is obtained by column swapping of $V_{n+1}$,
there exists $W_{n+1} \in \bZ^{2d \times 2d}$
such that $\det W_{n+1} = \pm 1$ and $V'_{n+1} = V_{n+1} W_{n+1}$.

Define $U_{n} = (u_{i,j})_{i,j=1}^d \in \bZ^{d \times d}$ as
\[
u_{i,j} = (-2)^{j-i} \binom{j-1}{i-1},
\]
where $\binom{j}{i}$ is a binomial coefficient and is defined to be zero if $i>j$.
Since $U_{n}$ is upper-triangular and all the diagonal entries are 1,
$U_{n} \in SL_d(\bZ)$ holds.
We now compute
$
V'_{n+1}
\begin{pmatrix}
U_{n} & O\\
O & U_{n}
\end{pmatrix}
$.
We have $\xi_{n+1,i+d} = -\xi_{n+1,i}$ for $1 \leq i \leq d$.
Further, using the formula $\cos{2\theta} = 2\cos^2 \theta -1$,
we have $\xi_{n,i} = \xi_{n+1,i}^2 -2 = \xi_{n+1,i+d}^2 -2$ for $1 \leq i \leq d$
and thus $\xi_{n,i}^a = (\xi_{n+1,i}^2 -2)^a = (\xi_{n+1,i+d}^2 -2)^a$ for all $a \in \bN$.
Thus we have 
\begin{align*}
V'_{n+1}
\begin{pmatrix}
U_{n} & O\\
O & U_{n}
\end{pmatrix}
=
\begin{pmatrix}
V_{n} & D_{n} V_{n} \\
V_{n} & -D_{n} V_{n}
\end{pmatrix}
.
\end{align*}

By induction assumption,
there exists $S_n \in \bZ^{d \times d}$ such that $\det S_n = \pm 1$ and $V_n S_n = A_n$.
Hence
\begin{align*}
\begin{pmatrix}
V_{n} & D_{n} V_{n} \\
V_{n} & -D_{n} V_{n}
\end{pmatrix}
\begin{pmatrix}
S_{n} & O \\
O & S_{n}
\end{pmatrix}
=
\begin{pmatrix}
A_{n} & D_{n} A_{n} \\
A_{n} & -D_{n} A_{n}
\end{pmatrix}
= A_{n+1}.
\end{align*}
Thus we have shown that $V_{n+1} S_{n+1} = A_{n+1}$
with
\[
S_{n+1} = W_{n+1}
\begin{pmatrix}
U_{n} & O\\
O & U_{n}
\end{pmatrix}
\begin{pmatrix}
S_{n} & O \\
O & S_{n}
\end{pmatrix}
.
\]
This shows that the lemma holds for $n+1$.
\end{proof}

\subsection{Recursive enumeration}
In this subsection we give a recursive algorithm to obtain the Chebyshev-Frolov lattice points.
We start with the definition of functions
which is used to state Lemma~\ref{lem:recursive-domain} to reduce a $2^{n+1}$-dimensional enumeration
to $2^n$-dimensional enumerations.
\begin{definition}
Let $n \in \bN$ and $d := 2^n$.
Let $\bsa_1, \bsb_1, \bsb_2, \bsc_1, \bsc_2 \in \bR^d$
and $\bsb = (\bsb_1; \bsb_2), \bsc := (\bsc_1; \bsc_2) \in \bR^{2d}$.
We define functions $\ave{n}{\bsb}$, $\Rb{n}{\bsa_1}{\bsb}{\bsc}$ and $\Rc{n}{\bsa_1}{\bsb}{\bsc}$ as 
\begin{align*} 
\ave{n}{\bsb} &= (\bsb_1 + \bsb_2)/2 \in \bR^d,\\
\Rb{n}{\bsa_1}{\bsb}{\bsc} &= D_n^{-1}\max(\bsb_1 - \bsa_1, -\bsc_2 + \bsa_1) \in \bR^d,\\
\Rc{n}{\bsa_1}{\bsb}{\bsc} &= D_n^{-1}\min(\bsc_1 - \bsa_1, -\bsb_2 + \bsa_1)  \in \bR^d.
\end{align*}
\end{definition}

\begin{lemma}\label{lem:recursive-domain}
Let $n \in \bN$ and put $d = 2^n$.
Let $\bsb_1, \bsb_2, \bsc_1, \bsc_2, \bsx_1, \bsx_2 \in \bR^d$
and define $\bsb, \bsc, \bsx \in \bR^{2d}$ as
 $\bsb = (\bsb_1; \bsb_2)$, $\bsc := (\bsc_1; \bsc_2)$ and $\bsx := (\bsx_1; \bsx_2)$.
Then the inequality
$\bsb \leq A_{n+1} \bsx \leq \bsc$
is equivalent to the simultaneous inequalities
\begin{numcases}{}
\ave{n}{\bsb} \leq A_n \bsx_1 \leq \ave{n}{\bsc}, \label{eq:equiv-1} \\
\Rb{n}{A_n \bsx_1}{\bsb}{\bsc} \leq A_n \bsx_2 \leq \Rc{n}{A_n \bsx_1}{\bsb}{\bsc}. \label{eq:equiv-2}
\end{numcases}
\end{lemma}

\begin{proof}
From Lemma~\ref{lem:matrix-recursive}, $\bsb \leq A_{n+1} \bsx \leq \bsc$ is equivalent to
\begin{equation}\label{eq:ineq-first}
\begin{cases}
\bsb_1 \leq A_n \bsx_1 + D_n A_n \bsx_2 \leq \bsc_1,  \\
\bsb_2 \leq A_n \bsx_1 - D_n A_n \bsx_2 \leq \bsc_2.
\end{cases}
\end{equation}
By adding the inequalities in \eqref{eq:ineq-first} we have
\begin{equation}\label{eq:ineq-add}
\ave{n}{\bsb} \leq A_n \bsx_1 \leq \ave{n}{\bsc}.
\end{equation}
On the other hand, 
\eqref{eq:ineq-first} is equivalent to
\begin{equation*}
\begin{cases}
\bsb_1 - A_n \bsx_1 \leq D_n A_n \bsx_2 \leq \bsc_1 - A_n \bsx_1, & \\
-\bsc_2 + A_n \bsx_1 \leq D_n A_n \bsx_2 \leq - \bsb_2 + A_n \bsx_1, &
\end{cases}
\end{equation*}
which is equivalent to 
\[
\max(\bsb_1 - A_n \bsx_1, -\bsc_2 + A_n \bsx_1) \leq D_n A_n \bsx_2
\leq \min(\bsc_1 - A_n \bsx_1, -\bsb_2 + A_n \bsx_1).
\]
Since $D_n$ is a diagonal matrix whose diagonal entries are positive, this inequality is equivalent to
\begin{equation}\label{eq:ineq-equiv}
\Rb{n}{A_n \bsx_1}{\bsb}{\bsc} \leq A_n \bsx_2 \leq \Rc{n}{A_n \bsx_1}{\bsb}{\bsc}.
\end{equation}
Thus we have
\[
\eqref{eq:ineq-first}
\iff \text{\eqref{eq:ineq-first} and \eqref{eq:ineq-add}}
\iff \text{\eqref{eq:ineq-equiv} and \eqref{eq:ineq-add}},
\]
which is what we desired to prove.
\end{proof}

\begin{remark}
Assume that $\bsb \leq \bsc$.
Then, for fixed $\bsx_1$ with \eqref{eq:equiv-1},
we can see that
$\Rb{n}{A_n \bsx_1}{\bsb}{\bsc} \leq \Rc{n}{A_n \bsx_1}{\bsb}{\bsc}$
and thus there exists $\bsx_2$ which satisfies \eqref{eq:ineq-equiv}.
\end{remark}

Let $n \in \bN$ and $\bsb, \bsc \in \bR^d$.
We define
\[
\Generatorshort{n}{\bsb}{\bsc}
:= \{\bsk \in \bZ^d \mid \bsb \leq A_n \bsk \leq \bsc \}.
\]
Lemma~\ref{lem:recursive-domain} implies the following theorem to give $\Generatorshort{n}{\bsb}{\bsc}$.

\begin{theorem}\label{thm:Frolov-recursive}
Let $n \in \bN$, $d := 2^n$ and $\bsb, \bsc \in \bR^{2d}$.
Then we have
\[
\Generatorshort{n+1}{\bsb}{\bsc}
= \left\{
\left(
\begin{aligned}
\bsk_1 \\
\bsk_2 
\end{aligned}
\right) \in \bR^{2d}
\;\middle|\,
\begin{aligned}
\bsk_1 &\in \Generatorshort{n}{\ave{n}{\bsb}}{\ave{n}{\bsc}},\\
\bsk_2 &\in \Generatorshort{n}{\Rb{n}{A_n \bsk_1}{\bsb}{\bsc}}{\Rc{n}{A_n \bsk_1}{\bsb}{\bsc}} 
\end{aligned}
\right\}.
\]
\end{theorem}

This theorem reduces
an enumeration in dimension $2^{n+1}$ to enumerations in dimension $2^n$.
Further the case $n=0$ is easy to solve,
since $k \in \Generatorshort{0}{b}{c}$ for $k \in \bZ$ and $b,c \in \bR$ is equivalent to $b \leq k \leq c$.
This justifies Algorithm~\ref{alg:recursive} to obtain the set $\Generatorshort{n}{\bsb}{\bsc}$.

\begin{algorithm}
\caption{Recursive algorithm to obtain the set $\Generatorshort{n}{\bsb}{\bsc}$}\label{alg:recursive}
\begin{algorithmic}[1]
\Procedure{Set}{$n,\bsb,\bsc$}\Comment{Output the set $\Generatorshort{n}{\bsb}{\bsc}$}
\If{$n=0$}
\State \textbf{return} $\{k \in \bZ \mid \lceil \bsb \rceil \leq k \leq \lfloor \bsc \rfloor \}$
\Comment{In this case $\bsb$ and $\bsc$ are scalar}
\Else
\State $P \gets \text{empty set}$ \Comment{Initialize $P$ as the empty set}
\ForAll{$\bsk_1 \in \Call{Set}{n-1 ,\ave{n-1}{\bsb}, \ave{n-1}{\bsc}}$}
\ForAll{$\bsk_2 \in \Call{Set}{n-1 , \Rb{n-1}{A_{n-1} \bsk_1}{\bsb}{\bsc}, \Rc{n-1}{A_{n-1} \bsk_1}{\bsb}{\bsc}}$}
\State \textbf{append} $P \gets (\bsk_1; \bsk_2)$ \Comment{Append a point to the set $P$}
\EndFor
\EndFor
\State \textbf{return} $P$
\EndIf
\EndProcedure
\end{algorithmic}
\end{algorithm}

\subsection{Sequential enumeration}
One disadvantage of Algorithm~\ref{alg:recursive} is that it needs much memory.
In this subsection, to defeat this disadvantage
we derive simultaneous inequalities equivalent to $\bsb \leq A_n \bsx \leq \bsc$
by applying Lemma~\ref{lem:recursive-domain} repeatedly
and then we give a sequential enumeration algorithm.

We begin with an illustration for the case $n=2$.
Fix $\bsb, \bsc \in \bR^4$ and let $\bsx = (x_1; x_2; x_3; x_4)$.
Our aim is to obtain simultaneous inequalities which are equivalent to $\bsb \leq A_2 \bsx \leq \bsc$.
From Lemma~\ref{lem:recursive-domain},
it is reduced to
\begin{numcases}{}
\Boxineq{1}{\bsbeta_{1,1}}{(x_1;x_2)}{\bsgamma_{1,1}}, & \label{eq:level1-0} \\
\Boxineq{1}{\bsbeta_{1,2}}{(x_3;x_4)}{\bsgamma_{1,2}}. & \label{eq:level1-1}
\end{numcases}
where we put $\bsbeta_{1,1} := \ave{1}{\bsb}$,
$\bsgamma_{1,1} := \ave{1}{\bsc}$,
$\bsbeta_{1,2} := \Rb{1}{A_1 (x_1;x_2)}{\bsb}{\bsc}$
and $\bsgamma_{1,2} := \Rc{1}{A_1 (x_1;x_2)}{\bsb}{\bsc}$.
Whereas $\bsbeta_{1,2}$ and $\bsgamma_{1,2}$ are not determined until $x_1$ and $x_2$ are fixed,
$\bsbeta_{1,1}$ and $\bsgamma_{1,1}$ are determined using only $\bsb$ and $\bsc$.
Hence we first consider \eqref{eq:level1-0}.
Again from Lemma~\ref{lem:recursive-domain}, \eqref{eq:level1-0} is reduced to
\begin{numcases}{}
\Boxineq{0}{\beta_{0,1}}{x_1}{\gamma_{0,1}}, & \label{eq:level2-0}\\
\Boxineq{0}{\beta_{0,2}}{x_2}{\gamma_{0,2}}, & \label{eq:level2-1}
\end{numcases}
where we put
$\beta_{0,1} := \ave{0}{\bsbeta_{1,1}}$,
$\gamma_{0,1} := \ave{0}{\bsgamma_{1,1}}$,
$\beta_{0,2} := \Rb{0}{A_0 x_1}{\bsbeta_{1,1}}{\bsgamma_{1,1}}$
and $\gamma_{0,2} := \Rc{0}{A_0 x_1}{\bsbeta_{1,1}}{\bsgamma_{1,1}}$.
Whereas $\beta_{0,2}$ and $\gamma_{0,2}$ are not determined until $x_1$ is fixed,
$\beta_{0,1}$ and $\gamma_{0,1}$ are determined using only $\bsb$ and $\bsc$.
Thus we can fix $x_1$ satisfying \eqref{eq:level2-0}.
Once $x_1$ is fixed, $\beta_{0,2}$ and $\gamma_{0,2}$ are determined
and thus we can fix $x_2$ with \eqref{eq:level2-1}.
Once $x_2$ is fixed, then $\bsbeta_{1,2}$ and $\bsgamma_{1,2}$ are determined,
and again from Lemma~\ref{lem:recursive-domain}, Inequality \eqref{eq:level1-1} is reduced to 
\begin{numcases}{}
\Boxineq{0}{\beta_{0,3}}{x_3}{\gamma_{0,3}}, & \label{eq:level2-2}\\
\Boxineq{0}{\beta_{0,4}}{x_4}{\gamma_{0,4}}, & \label{eq:level2-3}
\end{numcases}
where we put $\beta_{0,3} := \ave{0}{\bsbeta_{1,2}}$, $\gamma_{0,3} := \ave{0}{\bsgamma_{1,2}}$,
$\beta_{0,4} := \Rb{0}{A_0 x_3}{\bsbeta_{1,2}}{\bsgamma_{1,2}}$
and $\gamma_{0,4} := \Rc{0}{A_0 x_3}{\bsbeta_{1,2}}{\bsgamma_{1,2}}$.
Now $\beta_{0,3}$ and $\gamma_{0,3}$ are determined and we can fix $x_3$ with \eqref{eq:level2-2}.
Once $x_3$ is fixed, $\beta_{0,4}$ and $\gamma_{0,4}$ are determined
and thus we can fix $x_4$ with \eqref{eq:level2-3}.
In this way, we have shown that
$\bsb \leq A_2 \bsx \leq \bsc$ is equivalent to
the simultaneous inequalities 
\eqref{eq:level2-0}--\eqref{eq:level2-3},
where $\beta_{0,1}$ and $\gamma_{0,1}$ are already determined
and $\beta_{0,i}$ and $\gamma_{0,i}$ are determined when $x_1, \dots, x_{i-1}$ are fixed ($i=2,3,4$).
This equivalence allows us to implement the enumeration of the vectors
$\bsk \in \bZ^4$ with $\bsb \leq A_2 \bsk \leq \bsc$
by 4-folded for-loops or an equivalent tail-recursion.

We now generalize the procedure for any $n \in \bN$.
Hereafter, to clarify which coordinates we consider, we use the following notation.
\begin{definition}
Let $n, L, a \in \bN$ with $0 \leq L \leq n$, $1 \leq a \leq 2^{n-L}$
and $\bsb, \bsc \in \bR^d$.
Put $d' := 2^L$. We define
\begin{align*}
\bsxs{L}{a} &:= (x_{(a-1)d'+1}, \dots, x_{ad'})^\top \in \bZ^{d'},\\
\bsalphas{L}{a} &:= A_L \bsxs{L}{a} \in \bR^{d'}.
\end{align*}
\end{definition}

Put $d:=2^n$ and fix $\bsb, \bsc \in \bR^d$.
Our aim is to reduce $\bsb \leq A_n \bsxs{n}{1} \leq \bsc$
to simultaneous 1-dimensional inequalities.
Put $\bsbetas{n}{1} := \bsb$ and $\bsgammas{n}{1} :=\bsc$.
From Lemma~\ref{lem:recursive-domain},
for all $0 \leq L \leq n$, $1 \leq a \leq 2^{n-L}$
an inequality $\bsbetas{L}{a} \leq A_L \bsxs{L}{a} \leq \bsgammas{L}{a}$ is reduced to
\[
\begin{cases}
\bsbetas{L-1}{2a-1} \leq A_n \bsxs{L-1}{2a-1} \leq \bsgammas{L-1}{2a-1},\\
\bsbetas{L-1}{2a} \leq A_n \bsxs{L-1}{2a} \leq \bsgammas{L-1}{2a},
\end{cases}
\]
where $\bsbetas{L}{a}, \bsgammas{L}{a} \in \bR^{2^L}$ are defined as 
\begin{align}
\bsbetas{L-1}{2a-1}       &= \ave{L-1}{\bsbetas{L}{a}}, \label{eq:beta-left}\\
\bsgammas{L-1}{2a-1}    &= \ave{L-1}{\bsgammas{L}{a}}, \label{eq:gamma-left}\\
\bsbetas{L-1}{2a}    &= \Rb{L-1}{\bsalphas{L-1}{2a-1}}{\bsbetas{L}{a}}{\bsgammas{L}{a}}, \label{eq:beta-right}\\
\bsgammas{L-1}{2a} &= \Rc{L-1}{\bsalphas{L-1}{2a-1}}{\bsbetas{L}{a}}{\bsgammas{L}{a}}. \label{eq:gamma-right}
\end{align}
We have seen
that $\bsalphas{L}{a}$'s, $\bsbetas{L}{a}$'s and $\bsgammas{L}{a}$'s depend on each other
and some of them are not determined until some of $k_{i}$'s are fixed.
The dependence between $\bsalphas{L}{a}$'s are given as follows.
For $\bsalpha_1, \bsalpha_2 \in \bR^{2^L}$, define
\[
\UPa{L+1}{\bsalpha_1}{\bsalpha_2}
:= (\bsalpha_1 + D_L \bsalpha_2; \bsalpha_1 - D_L \bsalpha_2) \in \bR^{2^{L+1}}.
\]
Then for $1 \leq L \leq n$ and $1 \leq a \leq 2^{n-L}$ it follows from \eqref{eq:definition-A} that 
\begin{equation}\label{eq:alpha-update}
\bsalphas{L}{a} = \UPa{L}{\bsalphas{L-1}{2a-1}}{\bsalphas{L-1}{2a}}.
\end{equation}

We now study how those values are determined.
We define the sets $\Aset[i]$ and $\Bset[i]$ for $i \in \bN$, $0 \leq i \leq 2^n$
recursively as
\[
\Aset[0] := \emptyset, \qquad \Bset[0] := \{(j, 1) \mid j \in \bN, 0 \leq j \leq r\},
\]
and, for $i = 2^r p$ where $r \in \bN$ and $p$ is an odd integer,
\begin{align*}
\Aset[i] &= \Aset[i-1] \cup \{(j, 2^{r-j}p) \mid j \in \bN, 0 \leq j \leq r\}, \\
\Bset[i] &= \Bset[i-1] \cup \{(j, 2^{r-j}p+1) \mid j \in \bN, 0 \leq j \leq r\}.
\end{align*}
The following lemmas show that these sets control the determination of the values.

\begin{lemma}\label{lem:determine-alpha}
Let $i \in \bN$, $0 \leq i \leq 2^n$.
Let $x_1, \dots, x_i$ be fixed.
If $(L,a) \in \Aset$ holds, then $\bsalphas{L}{a}$ is determined.
\end{lemma}

\begin{proof}
We prove the lemma by induction on $i$.
If $i=0$, we have nothing to prove.
Now let $i= 2^r p >0$ where $r \in \bN$ and $p$ is an odd integer
and assume that the result holds for $i-1$.
Let $x_1, \dots, x_i$ be fixed.
By induction assumption, for all $(L,a) \in \Aset[i-1]$ the value $\bsalphas{L}{a}$ is determined.
Thus it remains to show for $(L,a) \in \Aset \setminus \Aset[i-1]$.
Since $x_i$ is fixed, $\bsalphas{0}{2^{r}p} = x_i$ is determined.
Further, by induction assumption,
for all $0 \leq j < r$ we have $(j, 2^{r-j}p-1) \in \Aset[2^r p -2^{j}] \subset \Aset[i-1]$
and thus $\bsalphas{j}{2^{r-j}p-1}$ is determined.
By using these results and applying \eqref{eq:alpha-update} repeatedly,
$\bsalphas{j}{2^{r-j}p}$ is sequentially determined for all $0 \leq j \leq r$.
This proves the result for $i$.
\end{proof}

We remark that the lemma is directly shown as follows:
The condition that $x_1, \dots, x_i$ are fixed implies 
that $\bsxs{L}{a}$ is fixed for all $(L,a) \in \Aset$ and thus
$\bsalphas{L}{a} = A_L \bsxs{L}{a}$ is determined.
The procedure shown in the proof, however, can save the cost to compute the values
similarly to the fast-Fourier transform algorithm.

\begin{lemma}\label{lem:determine-beta}
Let $i \in \bN$, $0 \leq i < 2^n$.
Let $x_1, \dots, x_i$ be fixed.
If $(L,a) \in \Bset$ holds, then $\bsbetas{L}{a}$ and $\bsgammas{L}{a}$ are determined.
\end{lemma}

\begin{proof}
We prove the lemma by induction on $i$.
First assume $i=0$, i.e., none of $x_j$ are fixed for $1 \leq j \leq 2^n$.
Even then, $\bsbetas{n}{1}$ and $\bsgammas{n}{1}$ are determined
as $\bsbetas{n}{1} = \bsb$ and $\bsgammas{n}{1} = \bsc$.
Hence, using \eqref{eq:beta-left} and \eqref{eq:gamma-left} repeatedly,
$\bsbetas{j}{1}$ and $\bsgammas{j}{1}$ are determined for all $0 \leq j \leq r$.
This proves the result for $i=0$.

Now we assume that the lemma holds for $i-1$.
Let $x_1, \dots, x_i$ be fixed.
By induction assumption,
$\bsbetas{L}{a}$ and $\bsgammas{L}{a}$ are determined for all $(L,a) \in \Bset[i-1]$.
Thus it remains to show for $(L,a) \in \Bset \setminus \Bset[i-1]$.
Lemma~\ref{lem:determine-alpha} implies that $\bsalphas{r}{p}$ is determined.
Further, by induction assumption we have $(r+1, (p+1)/2) \in \Bset[2^r(p-1)] \subset \Bset[i-1]$
and thus $\bsbetas{r+1}{(p+1)/2}$ and $\bsgammas{r+1}{(p+1)/2}$ are determined.
Then $\bsbetas{r}{p+1}$ and $\bsgammas{r}{p+1}$ are determined
from these results, \eqref{eq:beta-right} and \eqref{eq:gamma-right}.
Thus, by using \eqref{eq:beta-left} and \eqref{eq:gamma-left} repeatedly,
$\bsbetas{j}{2^{r-j}p+1}$ and $\bsgammas{j}{2^{r-j}p+1}$ are determined for all $0 \leq j \leq r$.
This proves the result for $i$.
\end{proof}
 
Lemma~\ref{lem:determine-beta} implies that $(0, i+1) \in \Bset[i]$,
which means that $\bsbetas{0}{i+1}$ and $\bsgammas{0}{i+1}$ are determined
when $x_1, \dots, x_i$ are fixed.
Thus we have shown the following equivalence in summary.

\begin{theorem}\label{thm:tail-recursive-domain}
The inequality $\bsb \leq A_n \bsx \leq \bsc$ is equivalent to
$2^n$ simultaneous inequalities
\[
\bsbetas{0}{i} \leq x_i \leq \bsgammas{0}{i} \quad \text{for $1 \leq i \leq 2^n$},
\]
where
$\bsbetas{0}{1}$ and $\bsgammas{0}{1}$ are already determined
and $\bsbetas{0}{i}$ and $\bsgammas{0}{i}$ are determined
when $x_1, \dots, x_{i-1}$ are fixed, as in Lemmas~\ref{lem:determine-alpha} and \ref{lem:determine-beta}.
\end{theorem}

Lemmas~\ref{lem:determine-alpha}--\ref{lem:determine-beta} and Theorem~\ref{thm:tail-recursive-domain} allow
Algorithm~\ref{alg:tail-recursive}, an tail-recursive enumeration
of all the Chebyshev-Frolov lattice points $A_n\bsk$ with $\bsk \in \bZ^{2^n}$ in the box $[\bsb,\bsc]$.
Algorithm~\ref{alg:tail-recursive} is equivalent to $2^n$-folded for-loops.

\begin{remark}
If you need only to approximate the integration value,
replace Line~22 in Algorithm~\ref{alg:tail-recursive} by the evaluation of the integrand.
You do not need to memorize any of the Chebyshev-Frolov lattice points.
\end{remark}

\begin{algorithm}
\caption{Enumerate the lattice points in the box $[\bsb,\bsc]$}\label{alg:tail-recursive}
\begin{algorithmic}[1]
\Procedure{LatticePoints}{$n, \bsb, \bsc$} \Comment{The lattice points in the box}
\For {$i = 1$ to $2^n$} \Comment{Preparation for updating}
\State memory $r(i), p(i) \in \bN$ as $i=2^{r(i)} p(i)$
\EndFor \Comment{Finish preparation}
\State $\bsbetas{n}{1} \gets \bsb$  \Comment{Update $\bsbetas{L}{a}$ and $\bsgammas{L}{a}$ with $\Bset[0]$}
\State $\bsgammas{n}{1} \gets \bsc$
\For {$j = n-1$ to $0$}
\State $\bsbetas{j}{1} \gets \ave{j}{\bsbetas{j+1}{1}}$
\State $\bsgammas{j}{1} \gets \ave{j}{\bsgammas{j+1}{1}}$
\EndFor \Comment{Finish updating}
\State $\Call{Enum}{1}$
\EndProcedure
\State
\Function{Enum}{$i$} \Comment{Enumerate the $i$-th coordinate $k_i$}
\For{$k_i = \lceil \bsbetas{0}{i} \rceil$ to $\lfloor \bsgammas{0}{i} \rfloor$}
\If{$i \neq 2^n$}
\State \Call{UpdateAlpha}{$i$}
\State \Call{UpdateBetaGamma}{$i$}
\State $\Call{Enum}{i+1}$
\Else \Comment{That is, if $i=2^n$}
\State \Call{UpdateAlpha}{$2^n$}
\State Output $\bsalphas{n}{1}$  \Comment{$\bsalphas{n}{1}=A_n \bsk$ is a lattice point}
\EndIf
\EndFor
\EndFunction
\State
\Function{UpdateAlpha}{$i$} \Comment{Update $\bsalphas{L}{a}$ with $\Aset$}
\State  $\bsalphas{0}{i} \gets k_i$
\For {$j = 1$ to $r(i)$}
\State $\bsalphas{j}{2^{r(i)-j}p(i)} \gets \UPa{j}{\bsalphas{j-1}{2^{r(i)-j+1}p(i)-1}}{\bsalphas{j-1}{2^{r(i)-j+1}p(i)}}$
\EndFor
\EndFunction
\State
\Function{UpdateBetaGamma}{$i$} \Comment{Update $\bsbetas{L}{a}$ and $\bsgammas{L}{a}$ with $\Bset$}
\State $\bsbetas{r(i)}{p(i)+1} \gets \Rb{r(i)}{\bsalphas{r(i)}{p(i)}}{\bsbetas{r(i)+1}{(p(i)+1)/2}}{\bsgammas{r(i)+1}{(p(i)+1)/2}}$
\State $\bsgammas{r(i)}{p(i)+1} \gets \Rc{r(i)}{\bsalphas{r(i)}{p(i)}}{\bsbetas{r(i)+1}{(p(i)+1)/2}}{\bsgammas{r(i)+1}{(p(i)+1)/2}}$
\For {$j = r(i)-1$ to $0$} 
\State $\bsbetas{j}{2^{r(i)-j}p(i)+1} \gets \ave{j}{\bsbetas{j+1}{2^{r(i)-j-1}p(i)+1}}$
\State $\bsgammas{j}{2^{r(i)-j}p(i)+1} \gets \ave{j}{\bsgammas{j+1}{2^{r(i)-j-1}p(i)+1}}$
\EndFor
\EndFunction
\end{algorithmic}
\end{algorithm}

\section{Frolov's cubature formula and its randomization}\label{sec:Frolov}
In this section we revisit Frolov's cubature formula and its randomization,
and in particular we show how to enumerate the integration nodes using Algorithm~\ref{alg:tail-recursive}.

Let $\bsv \in \bR^d$ and a matrix $T \in \bR^{d \times d}$ which generates an admissible lattice $T(\bZ^d)$.
We define the set
\[
X(T,\bsv) := \{ T(\bsk + \bsv) \mid \bsk \in \bZ^d\} \cap [-1/2, 1/2]^d
\]
and the cubature rule for a function $f(\bsx)$ on $[-1/2, 1/2]^d$ as
\[
Q_{T,\bsv}(f)
= \abs{\det{T}} \sum_{x \in X(T,v)} f(\bsx).
\]
As mentioned in the introduction,
Frolov's cubature formula is of the form $Q(a^{-1}T,\bszero)$ for $a > 1$.
For the number of integration nodes, it is known from \cite{Skriganov1994cud} that
\[
\lim_{a \to \infty} \frac{|X(a^{-1}T,\bszero)|}{\det(a^{-1}T)} =1.
\]

Following \cite{Kacwin2016ocf},
we use scaled (and coordinate-permuted) Chebyshev-Frolov lattices as admissible lattices for Frolov's cubature formula.
Let $n \in \bN$ and let $A_n$ be defined as in \eqref{eq:definition-A}.
For a scaling parameter $N \in \bR$ with $N>0$, we define the value $s(N):= (\abs{\det(A_n)} N)^{-1/d}$
and the matrix
\[
A_{n,N} := s(N) A_n,
\]
which satisfies $\abs{\det(A_{n,N})} = 1/N$.
From Theorem~\ref{thm:Kacwin}, we have $\abs{\det(A_n)} = (2d)^{d/2}/\sqrt{2}$.
We consider Frolov's cubature formula as $Q(A_{n,N},\bszero)$ for $N \in \bN$.
To find the integration nodes, we can use Algorithm~\ref{alg:tail-recursive} and the bijection
\[
\{A_n \bsk \mid \bsk \in \bZ^d\} \cap [\bsb, \bsc] \to X(A_{n,N}, \bszero),
\quad \bsx \mapsto s(N) \bsx,
\]
where $\bsb := -s(N)^{-1}(1/2, \dots, 1/2)^\top$ and $\bsc := -\bsb = s(N)^{-1}(1/2, \dots, 1/2)^\top$.

Randomized Frolov's cubature formula was introduced by Krieg and Novak \cite{Krieg2015uam},
and studied further by Ullrich \cite{Ullrich2016mcm}.
Our algorithm introduced below follows the exposition in \cite{Ullrich2016mcm},
but note that $A_{n,N}$ in this paper corresponds to $B_N^{- \top}$ in \cite{Ullrich2016mcm}.
Let $\bsu$ and $\bsv$ be two independent random vectors
that are uniformly distributed in $[1/2, 3/2]^d$ and $[0,1]^d$, respectively.
Let $U := \diag(\bsu)$.
We define randomized Frolov's cubature formula $M_N$ using $A_{n,N}$ as
\[
M_N(f) := Q_{U^{-1} A_{n,N}, \bsv}(f).
\]
How can we enumerate the nodes of the formula $M_N(f)$?
We have 
\begin{align*}
\bsx \in X(U^{-1} A_{n,N}, \bsv)
&\iff \bsx = U^{-1} A_{n,N} (\bsk + \bsv) \in [-1/2,1/2]^d\\
&\iff A_{n} \bsk \in s(N)^{-1}U[-1/2,1/2]^d - A_{n} \bsv.
\end{align*}
Hence, with defining $\bsh := (1/2, \dots, 1/2)^\top \in \bR^d$,
$\bsb = -s(N)^{-1}U \bsh -A_{n} \bsv$
and $\bsc = s(N)^{-1}U \bsh - A_{n} \bsv$,
we have the following bijective map
\[
\{A_n \bsk \mid \bsk \in \bZ^d\} \cap [\bsb, \bsc] \to X(U^{-1} A_{n,N}, \bsv),
\quad \bsx \mapsto s(N) U^{-1}(\bsx + A_{n} \bsv).
\]
Thus we can use Algorithm~\ref{alg:tail-recursive}
to enumerate the nodes of randomized Frolov's cubature formula.
We remark that the vector $A_{n} \bsv$ can be quickly computed,
similarly to the computation of $\bsalphas{n}{1}$.

\section{Numerical efficiency of the algorithm}\label{sec:experiments}
In this section we show the efficiency of Algorithm~\ref{alg:tail-recursive}.
Based on Algorithm~\ref{alg:tail-recursive},
we counted the number of the nodes of Frolov's cubature formula using the Chebyshev-Frolov lattices,
for dimensions $d=2,4,8,16,32$ and for the scaling parameter $N=2^m$ with $m=1, \dots, 30$.
More precisely, we replaced Line~22 in Algorithm~\ref{alg:tail-recursive} by incrementing a counter for the number of the nodes.
The code we used can be found at \url{https://github.com/tttyoyoyttt/the_Chebyshev_Frolov_lattice_points}.
We conducted the experiments on Intel Core i7-4790 3.60GHz CPU\@.
Our codes are implemented in C and compiled by GCC 4.9.3 with -O2 optimization flag on Windows 7.
We used windows.h for getting the execution time.

The result is summarized in Tables~\ref{table:d248}--\ref{table:d1632},
which show the number of the nodes and the execution time.
We can enumerate the nodes for $d=16$, $N=2^{20}$ in less than 1 second.
We can see that, for a fixed dimension $d$,
the execution time increases linearly with respect to the scaling parameter $N$.
On the other hand, for a fixed $N$, 
the execution time increases rapidly with respect to $d$.
We can also see that the scaling parameter $N$ does not well approximate the number of nodes
when $d=32$, for $N \leq 2^{30}$.
Hence we suggest to use the formula when $d \leq 16$.

We have to remark on the accuracy of Algorithm~\ref{alg:tail-recursive}.
It requires many floating-point arithmetic operations,
so it might have some errors.
Indeed, if we use single-precision instead of double-precision, 
the number of the enumerated points sometimes differs (for example, for $d=4$ and $N=2^{24}$).
In order to check the accuracy of Algorithm~\ref{alg:tail-recursive},
we conducted another experiment.
We first enumerated all the points in the box with the scaling parameter $2N$.
Then, for each point, we checked whether it is included in  the box with the scaling parameter $N$.
We confirm that the number of the enumerated points by this algorithm coincides with
that as in Tables~\ref{table:d248}--\ref{table:d1632}.
We also confirm that they coincide with the result in \cite[Appendix]{Nguyen2015cvs},
which gives those for $d=2,4,8,16$ and $N=4^m$ with $3 \leq m \leq 10$.
We remark that if we miscounted the integration nodes by the numerical error 
it would not be critical for Frolov's cubature formula
since such miscounted points would be very close to the edge of the box
and thus their function evaluations would be close to zero.

\begin{table}
\centering
\caption{The number of the nodes of Chebyshev-Frolov's cubature formula and the execution time
for $N=2^m$ with $m=1. \dots, 30$ and $d=2,4,8$ are given. We denote $\log_2 N$ by $\mathrm{lb} {N}$.}
\label{table:d248}
	\begin{tabular}{| l | l | l | l | l | l | l |}
	\hline&
	\multicolumn{2}{|c|}{$d=2$}&\multicolumn{2}{|c|}{$d=4$}&\multicolumn{2}{|c|}{$d=8$}\\\hline
	 $\mathrm{lb} {N}$ & nodes & time(sec) & nodes & time(sec) & nodes & time(sec) \\ \hline
	 $1$ & $3$ & $0.000008$& $5$ & $0.000012$& $19$ & $0.000019$\\\hline
	 $2$ & $5$ & $0.000001$& $5$ & $0.000001$& $19$ & $0.000005$\\\hline
	 $3$ & $7$ & $0.000001$& $11$ & $0.000002$& $23$ & $0.000007$\\\hline
	 $4$ & $15$ & $0.000001$& $15$ & $0.000002$& $27$ & $0.000008$\\\hline
	 $5$ & $31$ & $0.000001$& $31$ & $0.000003$& $45$ & $0.000014$\\\hline
	 $6$ & $65$ & $0.000001$& $71$ & $0.000004$& $79$ & $0.000021$\\\hline
	 $7$ & $131$ & $0.000001$& $123$ & $0.000006$& $167$ & $0.000034$\\\hline
	 $8$ & $257$ & $0.000002$& $261$ & $0.000009$& $271$ & $0.000053$\\\hline
	 $9$ & $513$ & $0.000005$& $513$ & $0.000016$& $529$ & $0.000092$\\\hline
	 $10$ & $1027$ & $0.000007$& $1025$ & $0.000028$& $1067$ & $0.000159$\\\hline
	 $11$ & $2049$ & $0.000012$& $2049$ & $0.000048$& $2107$ & $0.000282$\\\hline
	 $12$ & $4095$ & $0.000022$& $4099$ & $0.000083$& $4113$ & $0.000501$\\\hline
	 $13$ & $8191$ & $0.000042$& $8201$ & $0.000149$& $8283$ & $0.000891$\\\hline
	 $14$ & $16383$ & $0.000082$& $16385$ & $0.000272$& $16413$ & $0.001580$\\\hline
	 $15$ & $32767$ & $0.000160$& $32775$ & $0.000492$& $32823$ & $0.002782$\\\hline
	 $16$ & $65539$ & $0.000314$& $65533$ & $0.000910$& $65645$ & $0.005053$\\\hline
	 $17$ & $131075$ & $0.000610$& $131095$ & $0.001682$& $131183$ & $0.009033$\\\hline
	 $18$ & $262145$ & $0.001205$& $262143$ & $0.003110$& $262263$ & $0.016504$\\\hline
	 $19$ & $524289$ & $0.002393$& $524281$ & $0.005838$& $524341$ & $0.030049$\\\hline
	 $20$ & $1048579$ & $0.005123$& $1048609$ & $0.011248$& $1048779$ & $0.055015$\\\hline
	 $21$ & $2097153$ & $0.009583$& $2097143$ & $0.021416$& $2097107$ & $0.100984$\\\hline
	 $22$ & $4194307$ & $0.018736$& $4194355$ & $0.041369$& $4194399$ & $0.186696$\\\hline
	 $23$ & $8388611$ & $0.036209$& $8388589$ & $0.080187$& $8388843$ & $0.347014$\\\hline
	 $24$ & $16777215$ & $0.072445$& $16777221$ & $0.156289$& $16777535$ & $0.642543$\\\hline
	 $25$ & $33554429$ & $0.144748$& $33554439$ & $0.307664$& $33554807$ & $1.201434$\\\hline
	 $26$ & $67108861$ & $0.288265$& $67108867$ & $0.599280$& $67108777$ & $2.250724$\\\hline
	 $27$ & $134217727$ & $0.577065$& $134217723$ & $1.178398$& $134217783$ & $4.233333$\\\hline
	 $28$ & $268435457$ & $1.151628$& $268435461$ & $2.325093$& $268435889$ & $8.002449$\\\hline
	 $29$ & $536870913$ & $2.322550$& $536870913$ & $4.598802$& $536871467$ & $15.178914$\\\hline
	 $30$ & $1073741827$ & $4.603653$& $1073741807$ & $9.105867$& $1073742019$ & $28.899540$\\\hline
	\end{tabular}
\end{table}

\begin{table}
\centering
\caption{The number of the nodes of Chebyshev-Frolov's cubature formula and the execution time
for $N=2^m$ with $m=1. \dots, 30$ and $d=16,32$ are given. We denote $\log_2 N$ by $\mathrm{lb} {N}$.}
\label{table:d1632}
	\begin{tabular}{| l | l | l | l | l |}
	\hline&
	\multicolumn{2}{|c|}{$d=16$}&\multicolumn{2}{|c|}{$d=32$}\\\hline
	 $\mathrm{lb} {N}$ & nodes & time(sec) & nodes & time(sec) \\ \hline
	 $1$ & $77$ & $0.000085$& $3377$ & $0.019404$\\\hline
	 $2$ & $127$ & $0.000097$& $4105$ & $0.025259$\\\hline
	 $3$ & $151$ & $0.000120$& $5041$ & $0.034484$\\\hline
	 $4$ & $223$ & $0.000182$& $6371$ & $0.047148$\\\hline
	 $5$ & $295$ & $0.000260$& $8915$ & $0.068023$\\\hline
	 $6$ & $423$ & $0.000388$& $11867$ & $0.096804$\\\hline
	 $7$ & $539$ & $0.000569$& $15291$ & $0.141802$\\\hline
	 $8$ & $967$ & $0.000963$& $20651$ & $0.211015$\\\hline
	 $9$ & $1377$ & $0.001565$& $29215$ & $0.323648$\\\hline
	 $10$ & $2043$ & $0.002452$& $42323$ & $0.493032$\\\hline
	 $11$ & $3503$ & $0.004050$& $61997$ & $0.758875$\\\hline
	 $12$ & $5835$ & $0.007013$& $88645$ & $1.169238$\\\hline
	 $13$ & $10451$ & $0.011678$& $128269$ & $1.841059$\\\hline
	 $14$ & $18901$ & $0.020136$& $186749$ & $2.896117$\\\hline
	 $15$ & $36085$ & $0.034897$& $278961$ & $4.625915$\\\hline
	 $16$ & $69353$ & $0.060682$& $430037$ & $7.443627$\\\hline
	 $17$ & $136839$ & $0.107031$& $679287$ & $12.120136$\\\hline
	 $18$ & $267257$ & $0.188353$& $1102547$ & $20.047800$\\\hline
	 $19$ & $530333$ & $0.334023$& $1799443$ & $33.024893$\\\hline
	 $20$ & $1054837$ & $0.593132$& $2990409$ & $55.098279$\\\hline
	 $21$ & $2106165$ & $1.062962$& $5079585$ & $92.828770$\\\hline
	 $22$ & $4207997$ & $1.909237$& $8757305$ & $156.728010$\\\hline
	 $23$ & $8402385$ & $3.446778$& $15442557$ & $265.686566$\\\hline
	 $24$ & $16797845$ & $6.234337$& $27637841$ & $454.624777$\\\hline
	 $25$ & $33577467$ & $11.324212$& $50306689$ & $782.351191$\\\hline
	 $26$ & $67135425$ & $20.618313$& $92921093$ & $1351.141660$\\\hline
	 $27$ & $134246629$ & $37.640596$& $173897749$ & $2343.257467$\\\hline
	 $28$ & $268458047$ & $68.951316$& $328647641$ & $4087.063698$\\\hline
	 $29$ & $536891351$ & $126.640419$& $627372745$ & $7152.910717$\\\hline
	 $30$ & $1073829043$ & $233.271579$& $1208920345$ & $12553.435467$\\\hline
	\end{tabular}
\end{table}

\section*{Acknowledgments}
The authors are grateful to Christopher Kacwin, Mario Ullrich and Tino Ullrich for their valuable comments
about the theory and experiments for Frolov's cubature formula.

\bibliographystyle{plain}

\bibliography{Frolov}

\end{document}